\documentclass[11pt]{article}
\usepackage{amssymb}
\usepackage{amsmath}
\usepackage{amsthm}
\usepackage{graphicx}
\usepackage{fullpage}

\newtheorem{theorem}{Theorem}[section]

\newtheorem{observation}{Observation}[section]
\newtheorem{lemma}{Lemma}[section]
\newtheorem{definition}{Definition}[section]
\newtheorem{problem}{Problem}[section]
\newtheorem{construction}{Construction}[section]

\begin{document}

\title{Directed paths: from Ramsey to Ruzsa and Szemer\'edi}

\author{ Po-Shen Loh \thanks{Department of Mathematical
Sciences, Carnegie Mellon University, Pittsburgh, PA 15213. E-mail:
ploh@cmu.edu. Research supported in part by NSF grants DMS-1041500,
DMS-1201380, and DMS-1455125, and by a USA-Israel BSF Grant.} }

\date{}

\maketitle

\begin{abstract}
  Starting from an innocent Ramsey-theoretic question regarding directed
  paths in tournaments, we discover a series of rich and surprising
  connections that lead into the theory around a fundamental problem in
  Combinatorics: the Ruzsa-Szemer\'edi induced matching problem.  Using
  these relationships, we prove that every coloring of the edges of the
  transitive $n$-vertex tournament using three colors contains a directed
  path of length at least $\sqrt{n} \cdot e^{\log^* n}$ which entirely
  avoids some color.  We also expose connections to a family of
  constructions for Ramsey tournaments, and introduce and resolve some
  natural generalizations of the Ruzsa-Szemer\'edi problem which we
  encounter through our investigation.
\end{abstract}

\section{Introduction}

Ramsey theory is a central area in Combinatorics which concerns structure
that must exist in arbitrary (but large) configurations.  The classical
starting point for Graph Ramsey Theory considers the number of vertices
required in order for every 2-edge-coloring of a sufficiently large
complete graph to contain a monochromatic copy of a fixed complete graph
$K_t$ \cite{ErdSze, Ramsey}.  Over the century, many other target subgraphs
were considered, from the specific, such as paths \cite{GerGya}, to trees
and cliques \cite{Chvatal}, to entire classes such as graphs with bounded
degree \cite{ChvRodSzeTro, ConFoxSud, GraRodRuc} and degeneracy
\cite{BurErd}.

This paper explores an interesting Ramsey-type problem about paths, and so
we briefly outline some of the history.  In undirected graphs, an old
result of Gerencs\'er and Gy\'arf\'as \cite{GerGya} established that the
corresponding 2-color Ramsey number of an $n$-vertex path is precisely
$\lfloor \frac{3n-2}{2} \rfloor$.  When the number of colors increases
beyond two, however, less is known.  The result for three colors was proved
asymptotically by Figaj and \L uczak \cite{FigLuc}, and recently for large
$n$ by Gy\'arf\'as, Ruszink\'o, S\'ark\"ozy, and Szemer\'edi
\cite{GyaRusSarSze}.

Another flavor of Ramsey-type results, which also inspires the work in this
paper, stems from the following celebrated theorem of Erd\H{o}s and
Szekeres \cite{ErdSze}: every sequence of $n$ distinct real numbers
contains a monotone subsequence of length at least $\sqrt{n}$.  This is
often proven by the pigeonhole principle.  Indeed, assign to each number an
ordered pair $(x, y)$, where $x$ tracks the length of the longest
increasing subsequence ending at that number, and $y$ tracks the decreasing
analogue.  Since the numbers are distinct, it is easy to see that the
ordered pairs must also be distinct, and therefore some ordered pair has an
element at least $\sqrt{n}$.  Many extensions have been found for this
theorem (see, e.g., any of \cite{FoxPacSudSuk, Kalmanson, MosSha, Steele,
SzaTar}).

The same pigeonhole argument also proves the following straightforward
result.  Consider the transitive tournament on the vertex set $\{1, \ldots,
n\}$, where the edge between $i < j$ is oriented in the direction
$\overrightarrow{ij}$.  Then, every 2-coloring of the edges of $T_n$ has a
monochromatic (directed) path with at least $\sqrt{n}$ vertices (hereafter
called the \emph{vertex-length}). The generalization to monochromatic paths
in with $r$-colored tournaments (with sharp lower bound $n^{1/r}$) was
independently rediscovered multiple times \cite{Alo, Chv, GyaLeh}.  This
problem can also be restated in terms of the \emph{ordered Ramsey number}\/
of a path; see, e.g., any of \cite{BalCibKraKyn, ConFoxLeeSud,
FoxPacSudSuk, MilSchWes, MosSha} for a discussion of recent work in ordered
Ramsey Theory, several of which consider paths (in the context of
hypergraphs).

This paper discusses several questions that start from this classical
argument, leading to surprising and interesting connections across
Combinatorics.  This serendipitous tour passes through directed graph
Ramsey theory, to $k$-majority tournaments (applied to create constructions
for Ramsey tournaments), and ultimately to the borders of a foundational
problem of Ruzsa and Szemer\'edi on induced matchings.

Indeed, with two colors, the concept of ``monochromatic'' is unambiguous.
With three or more colors, however, there is another natural generalization
of the basic concept: one can seek substructures which are
\emph{1-color-avoiding}, i.e., in which there is at least one color which
does not appear at all.  (In the 2-color setting, these two concepts are
identical.)  Already in the 3-color setting, this problem is not
well-understood.

\begin{problem}
  \label{prb:3color-path}
  Determine $f(n)$, the maximum number such that every 3-coloring of the
  edges of the $n$-vertex transitive tournament contains a directed path
  with at least $f(n)$ vertices, which avoids at least one of the colors.
\end{problem}

By merging two of the three color classes, one may obtain a bound $f(n)
\geq \sqrt{n}$ using the 2-color result, and the following standard
construction achieves $f(n) \leq n^{2/3}$ when $n$ is a perfect cube:
identify the vertex set with the triples $\{1, \ldots, n\}^3$, ordered
lexicographically, and color an edge $(x, y, z) \rightarrow (x', y', z')$
based upon the leftmost position at which the triples differ.  The best
result, however, is still unknown.

It turns out that there are interesting connections between this problem
and a longstanding open question which traces back to a result of Ruzsa and
Szemer\'edi \cite{RuzSze} on induced matchings, which has rich connections
to Szemer\'edi's Regularity Lemma (see, e.g., the survey
\cite{KomShoSimSze}).

\begin{problem}
  \label{prb:RSz}
  Let $M(n, k)$ be the maximum number of edges in an $n$-vertex graph $G =
  (V, E)$ whose edge set is the union of $k$ induced matchings.  That is,
  $E = E_1 \cup \ldots \cup E_k$, where each $E_i$ is precisely the set of
  edges induced by $G[V_i]$ for some subset $V_i \subset V$, and each $E_i$
  is a matching.  Determine the asymptotics of $M(n, k)$.
\end{problem}

Leveraging this connection, and using the current best bounds on Problem
\ref{prb:RSz} (due to Fox's \cite{Fox} result on the Triangle Removal
Lemma), we obtain a non-trivial lower bound on $f(n)$.  Here, $\log^*$ is
the iterated logarithm, or the inverse of the tower function $T(n) =
2^{T(n-1)}$, $T(0) = 1$.  Also, we write $f(n) = \Omega(g(n))$ to mean that
there exists a positive constant $c$ such that $f(n) \geq c g(n)$ for all
sufficiently large $n$.

\begin{theorem}
  Every 3-coloring of the edges of the $n$-vertex transitive tournament
  contains a directed path of length at least $\Omega(\sqrt{n} \cdot
  e^{\log^* n})$, which is missing at least one color.
  \label{thm:f-lower}
\end{theorem}

Furthermore, we discover that Problem \ref{prb:3color-path} is related to
an ``ordered'' variant of the following generalization of the
Ruzsa-Szemer\'edi problem.

\begin{problem}
  Let $M_l(n, k)$ be the maximum number of edges in a $n$-vertex bipartite
  graph $G = (V, E)$ whose edge set is the union of $k$ matchings which are
  \textbf{$\boldsymbol{l}$-separated}.  That is, $E = E_1 \cup \ldots \cup
  E_k$, where each $E_i$ is a matching, and there is no sequence of
  vertices $a_0, a_1, \ldots, a_t$ with $t \leq l$, such that $a_0$ and
  $a_t$ are both incident to distinct edges from the same $E_i$, and each
  of $a_0 a_1$, $a_1 a_2$, \ldots, $a_{t-1} a_t$ are edges of $E$.
  Determine the asymptotics of $M_l(n, k)$.
  \label{prb:RSz-gen}
\end{problem}

Note that when $l = 1$, this corresponds to the induced matchings from
Problem \ref{prb:RSz}.  In our application, we are most interested in the
symmetric bipartite case, which corresponds to $M_l(2n, n)$.  We completely
resolve the unordered version stated above.  It turns out that already for
$l=2$, the exponent immediately drops, and remains level for all higher
$l$, so that the unique behavior is special to $l=1$.

\begin{theorem}
  For all positive integers $l \geq 2$ and $n$, we have $M_l(2n, n) =
  (1+o(1)) n^{3/2}$.
  \label{thm:floppy-sigma}
\end{theorem}

Finally, we show that the same bound on the ordered variant (formulated in
Section \ref{sec:RSz-gen-ord}) would improve Theorem \ref{thm:f-lower} to
achieve the asymptotically sharp bound $f(n) \geq \Omega(n^{2/3})$.

\bigskip

\section{Analysis}
\label{sec:1-avoid}

We start with directed paths that avoid at least one color.  In this
setting, the result is not even known for transitive tournaments, and so we
focus on that case here.  Recall that $f(N)$ is defined as the maximum
number $n$ such that every 3-coloring of the edges of the transitive
$N$-vertex tournament contains a directed path of vertex-length at least
$n$, which avoids at least one of the colors.

\subsection{Reformulation and $\boldsymbol{k}$-majority tournaments}

We begin by stating an equivalent and independently natural problem.  The
following statement is more convenient for relating with $k$-majority
tournaments, and this connection will be detailed at the end of this
section.

\begin{problem}
  \label{prb:triples}
  Determine $F(n)$, the maximum number of triples in a sequence $L_1, L_2,
  \ldots$ such that every coordinate is an integer between 1 and $n$
  inclusive, and for every $i < j$, there are at least two coordinates in
  which $L_i$ is strictly less than $L_j$.
\end{problem}

\noindent \textbf{Example.} The sequence $(3, 2, 7)$, $(7, 3, 2)$, $(1, 8,
8)$, $(2, 9, 9)$, \ldots\  is valid.  Indeed, when comparing the second and
fourth triples, note that although the first coordinate decreases from 7 to
2, the second and third coordinates both increase, from 3 to 8 and 2 to 8,
respectively.

\medskip

It turns out that Problem \ref{prb:3color-path} (Ramsey for
1-color-avoiding paths) and Problem \ref{prb:triples} (above) are
completely equivalent, as established by the following lemma.

\begin{lemma}
  For any positive integers $n$ and $N$: 
  \begin{description}
    \item[(i)] if $f(N) \leq n$, then $F(n) \geq N$; and
    \item[(ii)] if $f(N) > n$, then $F(n) < N$; and thus
    \item[(iii)] the values of $F$ are precisely the arguments at which $f$
      increases: $F(n) = N$ if and only if $f(N) = n$ and $f(N+1) > n$.
  \end{description}
  \label{lem:fF-equiv}
\end{lemma}

\begin{proof}
  For part (i), $f(N) \leq n$ means that there is a 3-coloring of the edges
  of the transitive tournament on vertex set $\{1, \ldots N\}$, with all
  1-color-avoiding paths having vertex-length at most $n$.  Consider this
  coloring, and follow the Erd\H{o}s-Szekeres proof.  Associate a triple of
  numbers to each vertex, with $L_k = (x_k, y_k, z_k)$ corresponding to the
  $k$-th vertex, where $x_k$ is the number of vertices in the longest path
  that ends at the $k$-th vertex and avoids the first color, and $y_k$ and
  $z_k$ are the corresponding statistics for avoiding the second and third
  colors, respectively.
  
  Consider any two triples $L_i$ and $L_j$ with $i < j$.  Since there are
  three colors overall, there are two colors which differ from the color of
  edge $\overrightarrow{ij}$.  Let $c$ be one of them.  Then, the longest
  path which avoids color $c$ and ends at the $i$-th vertex can be extended
  by edge $\overrightarrow{ij}$, to end at $j$.  Therefore, the
  corresponding coordinate (for color $c$) of $L_i$ must be strictly less
  than that of $L_j$.  This holds also for the other color which differs
  from $\overrightarrow{ij}$, and so $L_i$ is strictly less than $L_j$ in
  at least two coordinates.  We have verified the necessary properties
  produce a valid construction for Problem \ref{prb:triples} with $N$
  triples, all of whose entries are at most $n$.  This implies $F(n) \geq
  N$.

  To prove (ii), fix an arbitrary sequence of triples $L_1, L_2, \ldots,
  L_N$ with the property that for every $i < j$, there are at least two
  coordinates in which $L_i$ is strictly less than $L_j$, and all
  coordinates are positive integers.  We will show that some triple
  contains an integer strictly greater than $n$.  Indeed, let us construct
  a 3-edge-coloring of a transitive tournament on vertex set $\{1, \ldots,
  N\}$, where edges are directed in the natural forward increasing
  ordering.  For each pair $i < j$, there is always a choice of one of the
  three coordinates such that $L_i$ is strictly less than $L_j$ in the
  other two coordinates.  (If $L_i$ is strictly less than $L_j$ in all
  three coordinates, then there are three choices for this coordinate;
  disambiguate by selecting the first coordinate.) Use this coordinate as
  the color of the edge.  Since $f(N) > n$ by assumption, there is a
  1-color-avoiding directed path with more than $n$ vertices.

  At the same time, the following dynamic programming algorithm computes
  the length of the longest 1-color-avoiding directed path in a transitive
  tournament.  Iteratively construct triples $T_1, \ldots, T_N$, where each
  $T_k = (x_k, y_k, z_k)$ records the vertex-lengths of the longest
  1-color-avoiding paths ending at the $k$-th vertex, avoiding the first,
  second, and third color, respectively.  This can be achieved by starting
  with $T_1 = (1, 1, 1)$, and at each step, given $T_1, \ldots, T_{k-1}$,
  constructing $T_k$ by setting
  \begin{align*}
    x_k &= 1 + \max\{ x_j : 1 \leq j < k \text{ and }
    \overrightarrow{jk} \text{ not in color 1} \} \\
    y_k &= 1 + \max\{ y_j : 1 \leq j < k \text{ and }
    \overrightarrow{jk} \text{ not in color 2} \} \\
    z_k &= 1 + \max\{ z_j : 1 \leq j < k \text{ and }
    \overrightarrow{jk} \text{ not in color 3} \} ,
  \end{align*}
  with the convention that the maximum is taken to be 0 if it is taken over
  the empty set.  (This is actually consistent with the definition of
  $T_1$.)

  By construction of the tournament coloring, we inductively have that for
  each $k$, the triple $T_k$ is less than or equal to the given triple
  $L_k$ in every coordinate.  Since there is a 1-color-avoiding directed
  path with more than $n$ vertices, this implies that some entry of some
  $L_k$ exceeds $n$, as claimed.
  
  The final statement (iii) hinges on two observations: the sequence $f(1),
  f(2), f(3), \ldots$ is monotonic, and consecutive terms increase by at
  most one.  To see that the latter is true, consider an arbitrary
  3-edge-coloring of the transitive $N$-vertex tournament, and extend it to
  a 3-edge-coloring of the transitive $(N+1)$-vertex tournament by coloring
  all edges to the $(N+1)$-st vertex in the first color.  This would then
  contain a 1-color-avoiding path of vertex-length at least $f(N+1)$, which
  interacts with the $(N+1)$-st vertex in its final vertex, if at all.
  Therefore, the original 3-edge-coloring of the $N$-vertex tournament
  contained a 1-color-avoiding path of vertex-length at least $f(N+1) - 1$,
  hence $f(N) \geq f(N+1) - 1$, as claimed.

  Now suppose that $f(N) = n$ and $f(N+1) > n$.  By (i), we have $F(n) \geq
  N$, and by (ii), we have $F(n) < N+1$.  Since $F(n)$ is an integer, this
  implies $F(n) = N$.  Furthermore, since we have established that $f(1),
  f(2), \ldots$ is monotonic with consecutive terms increasing by at most
  one, and we also know that $f(1) = 1$ and the sequence is unbounded, this
  determines $F(n)$ for every positive integer $n$.  Since each $F(n)$ can
  only take one value, this establishes both directions of the final
  if-and-only-if statement.
\end{proof}

\bigskip

Problem \ref{prb:triples} is independently interesting in the above
formulation, as it connects to the area of $k$-majority tournaments.

\begin{definition}
  Given a set of $k$ linear orders on a common ground set $X$ (with $k$
  odd), the corresponding \textbf{$\boldsymbol{k}$-majority tournament} has
  vertex set $X$, with an edge directed from $x$ to $y$ if and only if a
  majority of the $k$ orders rank $x < y$.
\end{definition}

These types of tournaments have been actively studied in social science
\cite{McGarvey}, as well as in Combinatorics, where previous researchers
studied extremal problems regarding the sizes of dominating sets
\cite{AloBriKieKosWin} and transitive subtournaments \cite{MilSchWes}.
Problem \ref{prb:triples} is connected to bounding transitive
subtournaments in a natural class of constructions of $k$-majority
tournaments. Indeed, the author actually came upon the study of this
problem in the course of analyzing the following construction, which may
potentially be of independent interest.

\begin{construction}
  \label{cons:ramsey}
  Let $n$ and $k$ be positive integers, and let $\sigma_1, \ldots,
  \sigma_k$, be permutations of $\{1, \ldots, k\}$, such that $\sigma_i(1)
  = i$ for each $i$.  Define the vertex set $X = \{1, \ldots, n\}^k$ to be
  the collection of $k$-tuples of integers between 1 and $n$ inclusive.
  For each of the $k$ permutations $\sigma_i$, construct a linear order on
  $X$ by taking the lexicographical order where $\sigma_i(1)$ is the most
  significant coordinate, $\sigma_i(2)$ is the next most significant, and
  so on.  Finally, construct a $k$-majority tournament from these
  permutations.
\end{construction}

The central problem in constructing Ramsey tournaments is to avoid large
transitive subtournaments.  Strictly speaking, to estimate the size of the
largest transitive subtournament in this class of construction, one must
consider all elements of each permutation, because ties need to be broken.
However, when $n$ is large, the main asymptotic may be dictated by
transitive subtournaments which are entirely resolved by the first element
of every permutation (with no ties).  This type of transitive subtournament
is then a sequence of $k$-tuples in which every pair in the sequence
strictly increases in a majority of the coordinates.  When $k=3$, this is
precisely Problem \ref{prb:triples}.

\subsection{Induced matchings and Ruzsa-Szemer\'edi}

In light of the previous section, it suffices to focus on Problem
\ref{prb:triples}, which asks to upper-bound the lengths of sequences of
triples $L_1, L_2, \ldots$ from $\{1, 2, \ldots, n\}^3$, in which every
pair of triples increases in at least two coordinates.  The following
trivial observation implies that the sequence length is at most $n^2$.

\begin{observation}
  \label{obs:projection}
  For each fixed $x^*, y^*$, the sequence has at most one triple of the
  form $(x^*, y^*, z)$.
\end{observation}

To improve upon this trivial bound $F(n) \leq n^2$, we uncover a nice
connection to the celebrated Ruzsa-Szemer\'edi induced matching problem by
using the sequence of triples to construct the following bipartite graph
$G$.  Let the vertex set be $\{a_1, \ldots, a_n, b_1, \ldots, b_n\}$.  For
each triple $(x, y, z)$, place an edge between $a_x$ and $b_y$, and label
it with $z$.

\begin{lemma}
  For each label $z \in \{1, \ldots, n\}$, the set $E_z$ of edges labeled
  $z$ forms an \textbf{induced matching}: the subgraph of $G$ (which
  includes all edges of all labels) induced by all vertices spanned by
  those edges of $E_z$ is a matching.
  \label{lem:triples=>IM}
\end{lemma}

\begin{proof}
  Fix $z$.  By Observation \ref{obs:projection}, for each $x$, there is at
  most one triple of the form $(x, \star, z)$, and so each vertex $a_x$ is
  incident to at most one edge labeled $z$.  Similarly, each vertex $b_y$
  is incident to at most one edge labeled $z$.  Therefore, in the subgraph
  of $G$ induced by all vertices spanned by those edges of $E_z$, the edges
  of $E_z$ form a matching.
  
  It remains to show that no other edges (from other labels) are present in
  this induced subgraph.  Indeed, suppose for contradiction that vertices
  $a_x$ and $b_y$ are both incident to edges labeled $z$, but the edge $a_x
  b_y$ itself is labeled $z' \neq z$.  This means that the triple sequence
  contains the following triples, in some order (not necessarily the order
  presented):
  \begin{displaymath}
    (x, y', z)
  \end{displaymath}
  \begin{displaymath}
    (x', y, z)
  \end{displaymath}
  \begin{displaymath}
    (x, y, z')
  \end{displaymath}
  Without loss of generality, suppose that the first two triples come in
  the relative order presented.  Since both share the same third coordinate
  $z$, but must strictly increase in at least two coordinates, we have
  \begin{equation}
    \label{eq:get-IM}
    x < x' \quad \text{and} \quad y' < y .
  \end{equation}

  Similarly, since the first and third triples both agree in the first
  coordinate $x$, there must be a strict increase in both of the other two
  coordinates from one to the other.  Since we already have $y' < y$ from
  \eqref{eq:get-IM}, we must therefore also have $z < z'$.

  Finally, since the second and third triples both agree in the second
  coordinate $y$, there must be a strict increase in both of the other two
  coordinates from one to the other.  However, we already have $x' > x$ and
  $z < z'$, which are in opposite directions, which is a contradiction.
\end{proof}

We are now ready to prove our theorem.

\begin{proof}[Proof of Theorem \ref{thm:f-lower}]
  The induced matching property places us squarely in the context of
  Problem \ref{prb:RSz}, which studies $M(n, k)$, the maximum number of
  edges in an $n$-vertex graph whose edge set is the union of $k$ induced
  matchings.  Indeed, we immediately have $F(n) \leq M(2n, n)$, because our
  bipartite graph $G$ has exactly $F(n)$ edges.  The original result of
  Ruzsa and Szemer\'edi \cite{RuzSze} already breaks below the trivial
  quadratic bound, and the current best Triangle Removal Lemma bound of Fox
  \cite{Fox} gives $M(2n, n) \leq O( n^2 / e^{\log^* n} )$, which, in light
  of Lemma \ref{lem:fF-equiv}, proves our result.
\end{proof}

\subsection{Ruzsa-Szemer\'edi generalizations}
\label{sec:RSz-gen-ord}

It turns out that the bipartite graphs constructed from our triple
sequences have richer properties beyond the induced matching phenomena.
Indeed, the set of edges labeled with each fixed value $z$ forms an
``ordered'' induced matching, as illustrated in Figure
\ref{fig:ordered-matching}.

\begin{observation}
  \label{obs:ordered-matching}
  For each label $z \in \{1, \ldots, n\}$, the edges labeled $z$ are of the
  form $(a_{i_1}, b_{j_1})$, $(a_{i_2}, b_{j_2})$, \ldots, with $a_{i_1} <
  a_{i_2} < \cdots$ and $b_{j_1} < b_{j_2} < \cdots$.
\end{observation}

\begin{figure}[htpb]
  \begin{center}
    \includegraphics[height=2in]{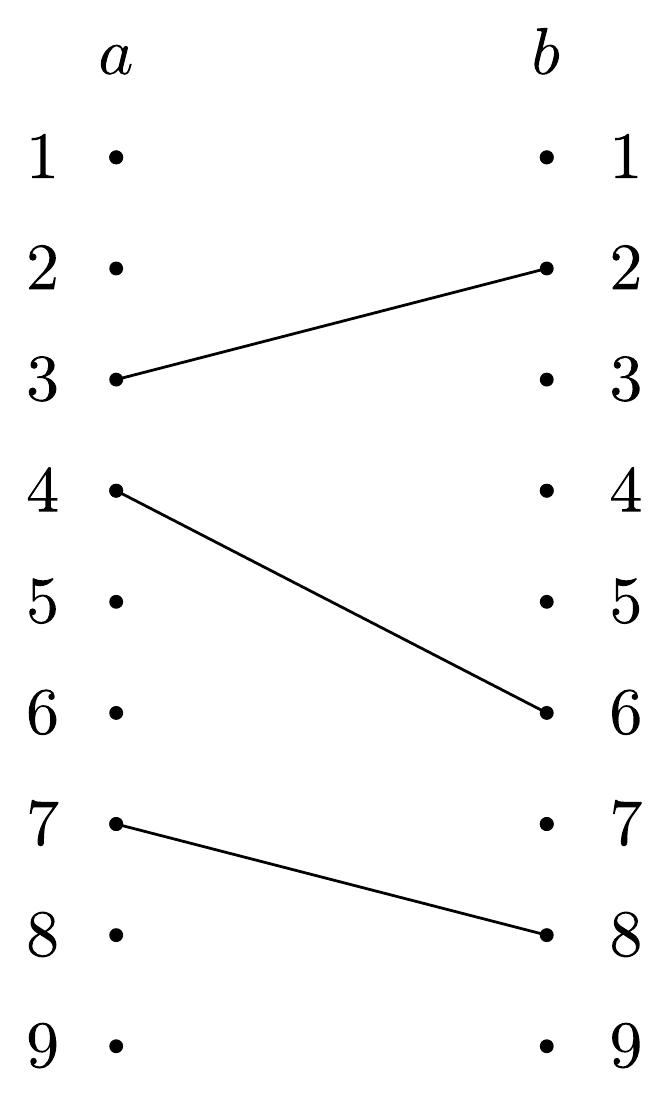}
  \end{center}
  \caption{Ordered matching.  Note that the matching edges do not cross
  each other.}
  \label{fig:ordered-matching}
\end{figure}

This additional ordering property alone is insufficient to significantly
improve the result, because Ruzsa and Szemer\'edi's lower bound
construction \cite{RuzSze}, based upon Behrend's \cite{Behrend, Elkin}
three-term arithmetic progression free sets, actually already produces
ordered induced matchings.  Their graph has about $n^2 / e^{\sqrt{\log n}}$
edges, so even if the additional order condition has an effect, it cannot
improve the exponent on $n$.  However, our collection of ordered induced
matchings still satisfies an additional condition, which we call
\emph{$\Sigma$-free} (illustrated in Figure \ref{fig:sigma-free}).

\begin{lemma}
  There are no vertices $b_h, a_i, b_j, a_k, b_l$ such that $i < k$ and $h
  \leq j \leq l$, edges $a_i b_h$ and $a_k b_l$ are both labeled $z$, and
  edges $a_i b_j$ and $a_k b_j$ are both in the graph (but may carry labels
  different from $z$).
  \label{lem:sigma-free}
\end{lemma}

\noindent \textbf{Remark.} By choosing $j=l$ or $h=j$, we recover the
ordered induced matching condition, so this is stronger than the
traditional context of Ruzsa-Szemer\'edi.

\begin{figure}[htpb]
  \begin{center}
    \includegraphics[height=2in]{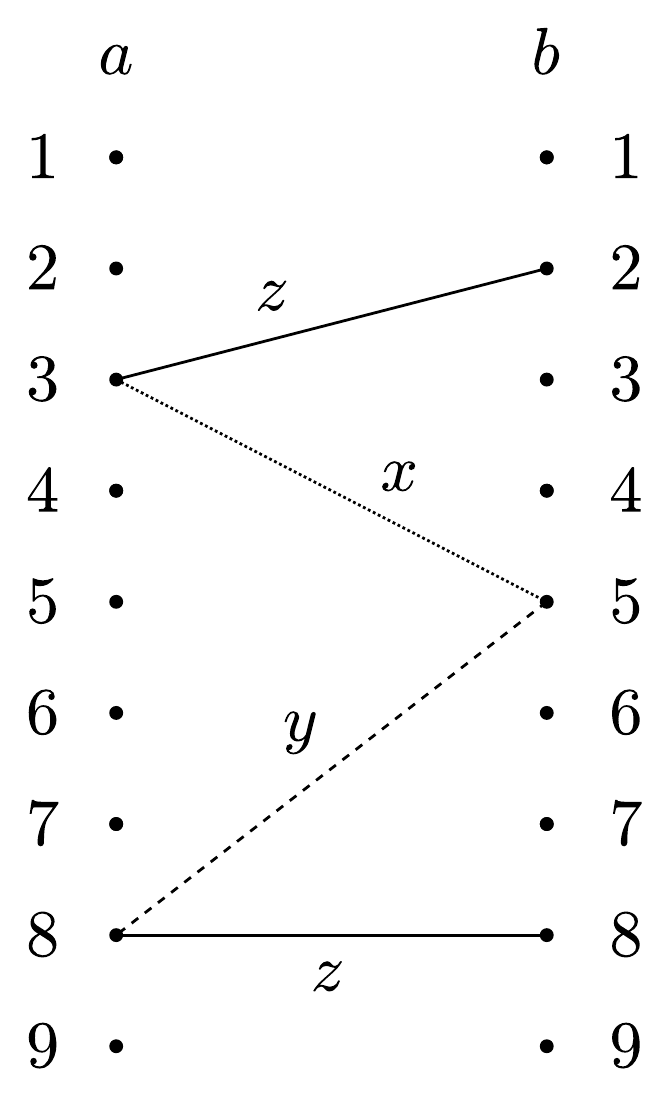}
  \end{center}
  \caption{Prohibited structure in $\Sigma$-free matchings.  The
    two outer edges both share the same label, but the connecting edges are
  differently labeled.}
  \label{fig:sigma-free}
\end{figure}

\begin{proof}[Proof of Lemma \ref{lem:sigma-free}]
  Lemma \ref{lem:triples=>IM} already established that the edge labels
  provide a partition into induced matchings, so we may assume that $h$,
  $j$, and $l$ are distinct.  Let $x$ be the label of the edge $a_i b_j$.
  We then have two triples $(i, h, z)$ and $(i, j, x)$ with $h < j$, so
  since one triple must increase in at least two coordinates, this forces
  $z < x$ as well.  Let $y$ be the label of the edge $a_k b_j$.  A similar
  argument considering $(i, j, x)$ and $(k, j, y)$ implies that $x < y$.
  However, the final edge $a_k b_l$ gives rise to $(k, l, z)$, and when we
  compare that triple to $(k, j, y)$, we find that they match in one
  coordinate, we already have deduced that $z < x < y$ and $j < l$, so
  neither can exceed the other in two coordinates.  This contradiction
  establishes the claim.
\end{proof}

\medskip

Lemma \ref{lem:sigma-free} now motivates Problem \ref{prb:RSz-gen} from the
introduction, which is a natural direction of generalization to the
Ruzsa-Szemer\'edi induced matching problem: rather than having two edges of
the same matching linked by a single edge of another matching, disallow
edges of the same matching to be linked by a path of length up to $l$
(possibly using edges of different labels).  Recall that we defined $M_l(n,
k)$ to be the maximum number of edges in a $n$-vertex bipartite graph whose
edge set is the union of $k$ matchings which are $l$-separated.

Lemma \ref{lem:sigma-free} actually provides an ``ordered'' condition which
is more restrictive than the forbidden structures in Problem
\ref{prb:RSz-gen}, so upper bounds on $M_l(2n, n)$ do not necessarily
transfer over to our original question.  However, due to its proximity to
the well-studied problem of Ruzsa and Szemer\'edi, Problem
\ref{prb:RSz-gen} is of independent interest.  We now proceed to prove the
sharp upper bound on $M_l(2n, n)$ for all $l \geq 2$.

\begin{proof}[Proof of Theorem \ref{thm:floppy-sigma}]
  We start by proving that $M_2(2n, n) \leq n^{3/2}$.  Suppose that we have
  a $(2n)$-vertex bipartite graph which is a disjoint union of 2-separated
  matchings $E_1, \ldots, E_n$.  Associate a set $S_v$ of indices $i$ to
  each vertex $v$, which records which $E_i$ contain edges incident to $v$.
  Now, if there is a vertex $v$ such that two distinct neighbors $x$ and
  $y$ are both incident to edges of the same $E_i$, then the 2-separation
  condition is violated.  (Here we use the assumption that the graph is
  bipartite: they cannot both be incident to the same edge.) Thus, for each
  vertex $w$, its neighbors have disjoint $S_v$.
  
  Since each $E_i$ is a matching, each vertex is incident to at most one
  edge of each $E_i$.  Therefore, for every vertex, the sum of the degrees
  of its neighbors is at most the number of matchings $n$; summing over all
  $2n$ vertices, the number of two-edge walks in the graph is at most
  $(2n)(n) = 2n^2$.  Since this is also precisely the sum of the squares of
  the degrees, convexity implies
  \begin{displaymath}
    2n^2 = \sum_v d_v^2 \geq (2n) (\overline{d})^2,
  \end{displaymath}
  where $\overline{d}$ is the average degree.  Therefore, the average
  degree is at most $\sqrt{n}$, which implies that there are at most
  $\frac{1}{2} (2n) \sqrt{n}$ edges, and $M_2(2n, n) \leq n^{3/2}$ as
  claimed.  Since $M_l(2n, n)$ is clearly monotonic in $l$, this then
  immediately implies that $M_l(2n, n) \leq n^{3/2}$ for all $l \geq 2$ as
  well.

  For the lower bound, we will prove that this bound is exactly sharp for
  every perfect square $n$ through an explicit construction.  Start with
  the disjoint union of $\sqrt{n}$ copies of the complete bipartite graph
  $K_{\sqrt{n}, \sqrt{n}}$.  Each component has exactly $n$ edges; assign
  one to each $E_i$.  Now, every pair of distinct edges from the same $E_i$
  actually comes from two different connected components, and are therefore
  at infinite distance in the graph.  Hence, for every perfect square $n$
  and every $l \geq 2$, we have $M_l(2n, n) = n^{3/2}$, which completes the
  proof.
\end{proof}

\medskip

\noindent \textbf{Remark.} If the bound in Theorem \ref{thm:floppy-sigma}
could be established in the weaker $\Sigma$-free setting, we would
immediately achieve $F(n) \leq n^{3/2}$, which by Lemma \ref{lem:fF-equiv}
would produce the sharp bound $f(n) \geq n^{2/3}$ on our original Ramsey
question on 1-color-avoiding paths.

\section{Concluding remarks}

This study has exposed many interesting problems, ranging from Ramsey
theory, to digraphs, and on to Regularity.  The original open problem is to
improve the Ramsey bounds for 1-color-avoiding paths in 3-edge-colorings of
transitive tournaments (Problem \ref{prb:3color-path}).  It is also
interesting to consider general tournaments.  There may also be interesting
phenomena when more than three colors are involved.  Indeed, there is a
whole spectrum of Ramsey-type problems where the notion of monochromatic
slides across the scale from 1-color to 1-color-avoiding.

There are also interesting new questions which border on the area of
Szemer\'edi's Regularity Lemma.  Our Theorem \ref{thm:floppy-sigma} handles
(unordered) $l$-separated matchings.  It would be interesting to prove the
analogue of the $M_2(2n, n)$ upper bound for $\Sigma$-free matchings (which
are ordered).  The $\Sigma$-free condition naturally generalizes to ordered
notions of $l$-separation between matchings, and indeed, our graphs satisfy
all properties along this successively richer hierarchy.  Any improvement
anywhere along this hierarchy would directly translate into improvements
for our problem.

Finally, it would be interesting to analyze the Ramsey tournaments produced
by well-chosen collections of permutations in the class of Construction
\ref{cons:ramsey}. If the transitive subtournaments are indeed small, this
could give rise to new alternative constructions in that area of study.

\section*{Acknowledgments}

The author would like to thank Jacob Fox for interesting discussions on the
topic of $k$-majority tournaments, which were the actual inspiration for
this investigation. He would also like to thank Dhruv Mubayi for helpful
conversations, including tracking down earlier references.  Finally,
several student staff members from the USA Math Olympiad Summer Program
were involved in early work regarding the formulation in Problem
\ref{prb:triples}, including Linus Hamilton, Jenny Iglesias, and Matt
Superdock.


\begin{thebibliography}{99}

  \bibitem{Alo} N. Alon, Monochromatic directed walks in arc colored
    directed graphs, \emph{Acta Math. Acad. Sci. Hungar.} \textbf{49}
    (1987), 163--167.

  \bibitem{AloBriKieKosWin} N. Alon, G. Brightwell, H. Kierstead, A.
    Kostochka, and P. Winkler, Dominating sets in $k$-majority tournaments,
    \emph{J. Combin. Theory Ser. B} \textbf{96} (2006), 374--387.

  \bibitem{BalCibKraKyn} M. Balko, J. Cibulka, K. Kr\'al, and J. Kyn\v{c}l,
    Ramsey numbers of ordered graphs, submitted.

  \bibitem{Behrend} F. Behrend, On sets of integers which contain no three
    terms in arithmetic progression, \emph{Proc. Nat. Acad. Sci.}
    \textbf{32} (1946), 331--332.

  \bibitem{BurErd} S. Burr and P. Erd\H{o}s, On the magnitude of
    generalized Ramsey numbers for graphs, in Infinite and Finite Sets,
    Vol. 1 (Keszthely, 1973), 214--240, Colloq. Math. Soc. J\'anos Bolyai,
    Vol. 10, North-Holland, Amsterdam, 1975.

  \bibitem{Chv} V. Chvatal, Monochromatic paths in edge-colored graphs,
    \emph{J.  Combinatorial Theory Ser. B} \textbf{13} (1972), 69--70,

  \bibitem{Chvatal} V. Chv\'atal, Tree-complete graph Ramsey numbers,
    \emph{J. Graph Theory} \textbf{1} (1977), 93.

  \bibitem{ChvRodSzeTro} V. Chv\'atal, V. R\"odl, E. Szemer\'edi, and W.
    Trotter, The Ramsey number of a graph with bounded maximum degree,
    \emph{J. Comb. Theory Ser. B} \textbf{34} (1983), 239--243.

  \bibitem{ConFoxLeeSud} D. Conlon, J. Fox, C. Lee, and B. Sudakov, Ordered
    Ramsey numbers, submitted.

  \bibitem{ConFoxSud} D. Conlon, J. Fox, and B. Sudakov,
    \emph{Combinatorica} \textbf{32} (2012), 513--535.

  \bibitem{Elkin} M. Elkin, An improved construction of progression-free
    sets, \emph{Israel Journal of Mathematics} \textbf{184} (2011),
    93--128.

  \bibitem{ErdSze} P.~Erd\H{o}s and G.~Szekeres, A combinatorial problem in
    geometry, {\em Compositio Math.} \textbf{2} (1935), 463--470.

  \bibitem{FigLuc} A. Figaj and T. \L uczak, The Ramsey number for a triple
    of long even cycles, \emph{J. Combin. Theory Ser. B} \textbf{97}
    (2007), 584--596.

  \bibitem{Fox} J. Fox, A new proof of the graph removal lemma, \emph{Ann.
    of Math} \textbf{174} (2011), 561--579.

  \bibitem{FoxPacSudSuk} J.~Fox, J.~Pach, B.~Sudakov, and A.~Suk,
    Erd\H{o}s--Szekeres-type theorems for monotone paths and convex bodies,
    {\em Proceedings of the London Mathematical Society}, \textbf{105}
    (2012), 953--982.

  \bibitem{GerGya} L. Gerencs\'er and Gy\'arf\'as, On Ramsey-type
    problems, \emph{Annales Universitatis Scientiarum Budapestinensis,
    E\"otv\"os Sect. Math.} \textbf{10} (1967), 167--170.

  \bibitem{GraRodRuc} R. Graham, V. R\"odl, and A. Ruci\'nski, On graphs
    with linear Ramsey numbers, \emph{J. Graph Theory} \textbf{35} (2000),
    176--192.

  \bibitem{GyaLeh} A. Gy\'arf\'as and J. Lehel, A Ramsey-type problem in
    directed and bipartite graphs, \emph{Period. Math. Hungar.} \textbf{3}
    (1973), 299--304.

  \bibitem{GyaRusSarSze} A. Gy\'arf\'as, M. Ruszink\'o, S\'ark\"ozy, and E.
    Szemer\'edi, Three color Ramsey numbers for paths, \emph{Combinatorica}
    \textbf{27} (2007), 35--69.

  \bibitem{Kalmanson} K.~Kalmanson, On a theorem of Erd\H{o}s and Szekeres,
    {\em Journal of Combinatorial Theory, Series A}, \textbf{15} (1973),
    343--346.

  \bibitem{KomShoSimSze} J. Koml\'os, A. Shokoufandeh, M. Simonovits, and
    E.  Szemer\'edi, The regularity lemma and its applications in graph
    theory, in: Theoretical aspects of computer science (Tehran, 2000),
    Lecture Notes in Comput. Sci., 2292, Springer, Berlin, 2002, 84--112.

  \bibitem{McGarvey} D. McGarvey, A theorem on the construction of voting
    paradoxes, \emph{Econometrica} \textbf{21} (1953), 608--610.

  \bibitem{MilSchWes} K. Milans, D. Schreiber, and D. West, Acyclic sets in
    $k$-majority tournaments, \emph{Electron. J. Combin.} \textbf{18}
    (2011), P122.

  \bibitem{MosSha} G.~Moshkovitz and A.~Shapira, Ramsey theory, integer
    partitions and a new proof of the Erd\H{o}s--Szekeres Theorem,
    \emph{Advances in Math.} \textbf{262} (2014), 1107--1129.

  \bibitem{Ramsey} F. Ramsey, On a problem in formal logic, \emph{Proc.
    London Math Soc.} \textbf{30} (1930), 264--286.

  \bibitem{RuzSze} I. Z. Ruzsa and E. Szemer\'edi, Triple systems with no
    six points carrying three triangles, \emph{Proc.  Colloq. Math.  Soc.
    J\'anos Bolyai} \textbf{18} (1978), 939--945.

  \bibitem{Steele} M.~Steele, Variations on the monotone subsequence theme
    of {E}rd\H{o}s and {S}zekeres, In {\em Discrete probability and
    algorithms}, Springer (1995), 111--131.

  \bibitem{SzaTar} T.~Szab\'o and G.~Tardos, A multidimensional
    generalization of the Erd\H{o}s--Szekeres lemma on monotone
    subsequences, {\em Combinatorics, Probability and Computing}
    \textbf{10} (2001), 557--565.

\end{thebibliography}
\end{document}